\documentclass[11pt]{article}
    
\usepackage{amsmath,amsfonts,amsthm,amscd,amssymb,graphicx}

\numberwithin{equation}{section}

\usepackage{hyperref}

\newtheorem{theorem}{Theorem}[section]

\newtheorem{lemma}[theorem]{Lemma}

\newtheorem{prop}[theorem]{Proposition}

\newcommand{\RR}{\mathbb{R}}
\newcommand{\cO}{\mathcal{O}}

\newcommand{\ZZ}{{\mathbb Z}}

\def\beq{\begin{equation}}
\def\eeq{\end{equation}}
\def\bb1{{1\!\!1}}
\def\rit{{\Bbb R}}

\def\tit{{\Bbb T}}

\def\bl{\mathrm{bl}}
\def\app{\mathrm{app}}

\begin{document}

\title{On nonlinear instability of Prandtl's boundary layers: the case of Rayleigh's stable shear flows} 

\author{Emmanuel Grenier\footnotemark[1]
  \and Toan T. Nguyen\footnotemark[2]
}

\maketitle

\renewcommand{\thefootnote}{\fnsymbol{footnote}}

\footnotetext[1]{Equipe Projet Inria NUMED,
 INRIA Rh\^one Alpes, Unit\'e de Math\'ematiques Pures et Appliqu\'ees., 
 UMR 5669, CNRS et \'Ecole Normale Sup\'erieure de Lyon,
               46, all\'ee d'Italie, 69364 Lyon Cedex 07, France. Email: Emmanuel.Grenier@ens-lyon.fr}

\footnotetext[2]{Department of Mathematics, Penn State University, State College, PA 16803. Email: nguyen@math.psu.edu. 
TN's research was supported in part by the NSF under grant DMS 1764119.}




\begin{abstract}

In 1904, Prandtl introduced his famous boundary layer in order to describe the behavior of solutions of Navier Stokes equations
near a boundary as the viscosity goes to $0$. His Ansatz has later  been  justified for analytic data by R.E. Caflisch and 
M. Sammartino. In this paper, we prove that his expansion is false, up to $O(\nu^{1/4})$ order terms in $L^\infty$ norm,
 in the case of solutions with Sobolev regularity, 
even in cases where the Prandlt's equation is well posed in Sobolev spaces. 

In addition, we also prove that  monotonic boundary layer profiles, which are stable when $\nu = 0$, 
are nonlinearly unstable when $\nu > 0$, provided $\nu$ is small enough, 
up to $O(\nu^{1/4})$ terms in $L^\infty$ norm.  

\end{abstract}

\begin{abstract}

En $1905$, Prandtl introduit sa c\'el\`ebre couche limite pour d\'ecrire le comportement des solutions des \'equations
de Navier Stokes pr\`es d'un bord lorsque la viscosit\'e tend vers $0$. Son Ansatz a \'et\'e justifi\'e dans le cas
de donn\'ees analytiques par R.E. Caflisch et M. Sammartino. Dans cet article nous prouvons que ce d\'eveloppement
asymptotique est faux dans $L^\infty$ \`a l'ordre $O(\nu^{1/4})$ pour des solutions \`a r\'egularit\'e Sobolev,
m\^eme dans des cas o\`u l'\'equation de Prandtl est elle bien pos\'ee.

De plus nous prouvons aussi que des profils de cisaillement monotones, qui sont stables \`a $\nu = 0$,
sont nonlin\'eairement instable quand $\nu > 0$ est assez petit, \`a l'order $O(\nu^{1/4})$ dans $L^\infty$.

\end{abstract}


\section{Introduction}


In this paper, we are interested in the inviscid limit $\nu \to 0$ of the Navier-Stokes equations for incompressible fluids, possibly subject to some external forcing $f^\nu$,
namely
\beq \label{NS1}
\partial_t u^\nu + (u^\nu \cdot \nabla) u^\nu + \nabla p^\nu  = \nu \Delta u^\nu + f^\nu,
\eeq
\beq \label{NS2} 
\nabla \cdot u^\nu = 0,
\eeq
on the half plane $\Omega = \{(x,y) \in \tit \times \rit^+\}$, with the no-slip boundary condition
\beq \label{NS3} 
u^\nu = 0 \quad \hbox{on} \quad \partial \Omega.
\eeq
As $\nu$ goes to $0$, one would expect the solutions $u^\nu$ to converge to solutions of Euler equations for incompressible fluids
\beq \label{Euler1}
\partial_t u^0 + (u^0 \cdot \nabla) u^0 + \nabla p^0 = f^0,
\eeq
\beq \label{Euler2}
\nabla \cdot u^0 = 0,
\eeq
with the boundary condition 
\beq \label{Euler3}
u^0 \cdot n = 0 \quad \hbox{on} \quad \partial \Omega,
\eeq
where $n$ is the unit normal to $\partial \Omega$.

At the beginning of the twentieth century, Prandtl introduced his well known boundary layers in order to describe the
transition from Navier-Stokes to Euler equations as the viscosity tends to zero. Formally, we expect that
\beq \label{Ansatz}
u^\nu(t,x,y) \approx  u^0(t,x,y) + u_P \Bigl(t,x, {y \over \sqrt{\nu}}\Bigr) + \mathcal{O}(\sqrt{\nu})_{L^\infty}
\eeq
where $u^0$ solves the Euler equations \eqref{Euler1}-\eqref{Euler3}, and $u_P$ is the Prandtl boundary layer correction, 
which is of order one.
The size of Prandtl's boundary layer  is of order $\sqrt{\nu}$. Formally it is even possible to write an asymptotic
expansion for $u^\nu$ in terms of powers of $\sqrt{\nu}$. 
The aim of this paper is to investigate whether (\ref{Ansatz}) holds true. 

The Prandtl's boundary layer equations on $u_P = (u_{P,1},u_{P,2})$ read
\beq\label{Pr} 
\begin{aligned}
\partial_t u_{P,1} + u_P \cdot \nabla u_{P,1} + \partial_x p^0 &= \partial_z^2 u_{P,1} + f^P,
\\
\nabla \cdot u_P &=0, 
\end{aligned}
\end{equation}
together with appropriate boundary conditions to correct the no-slip boundary conditions of Navier-Stokes solutions.
 In the above, $\partial_x p^0$ denotes the pressure gradient of the Euler flow on the boundary, and $u_{P,1}$ is the horizontal 
 component of the velocity.  

These equations have been intensively studied in the mathematical literature. Notably, solutions to the Prandtl equations 
have been constructed for monotonic data \cite{Oleinik1,Oleinik, Xu, MW} or data with Gevrey or analytic regularity \cite{Caf1, DGV1, VV}. 
In the case of non-monotonic data with Sobolev regularity, the Prandtl equations are ill-posed  \cite{DGV,GVN,GuoN}. 

The validity of Prandtl's Ansatz \eqref{Ansatz} has been established in \cite{Caf1,Caf2}
for initial data with analytic regularity, leaving a remainder of order $\sqrt \nu$. A similar result is also obtained in \cite{Mae}. 
The Ansatz \eqref{Ansatz}, with a specific boundary layer profile, has been recently justified for data with Gevrey regularity \cite{DGV2} 
by deriving sharp linear semigroup bounds near a stationary shear flow via energy estimates. 

However these positive results hide a strong instability occurring at high spatial frequencies. 
For some profiles, there exist instabilities with horizontal wave numbers of order $\nu^{-1/2}$ 
which grow like $\exp(C t / \sqrt{\nu})$. Within an analytic framework, these
instabilities are initially of order $\exp(-D / \sqrt{\nu})$ and therefore later are of size $\exp( (Ct - D) /\sqrt{\nu})$. 
They remain negligible in bounded time (as long as $t < D/ 2 C$ for instance).

Within Sobolev spaces, these instabilities are predominant. Initially they are of size $C \nu^{-s/2}$
and grow like $C \nu^{-s/2} \exp( C t / \sqrt{\nu})$. They reach $O(1)$ within vanishing times, of order
$\sqrt{\nu} \log \nu$.
 
To construct instabilities we focus on particular initial data, called
shear layer profiles, namely initial data of the form
$$
u^\nu(0) = (U_0(\sqrt{\nu}),0) .
$$
The function $U_0$ is called the profile of the shear layer.
Two kinds of instability may be described

\begin{itemize}

\item An instability with horizontal wave numbers of order $O(\nu^{-1/2})$, growing like $\exp(C t / \sqrt{\nu})$.
This instability occurs when the profile is unstable with respect to the inviscid Euler equations, or, very roughly speaking, 
when the profile $U_0$ has a "strong" inflexion point (in the spirit of Rayleigh's criterium of unstability). For such profiles 
E. Grenier proved in \cite{Gr1} that Prandtl's asymptotic expansion is false,  
up to a remainder of order $\nu^{1/4}$ in $L^\infty$ norm. 
More recently, E. Grenier and T. Nguyen managed to replace $O(\nu^{1/4})$ by $O(1)$ and proved that the difference between
the genuine solution and the Prandtl's expansion may be of order one in supremum norm, and that this difference does not
vanish as $\nu$ goes to $0$. These instabilities are driven by inviscid instabilities occurring within the boundary layer.

For such shear layer profiles, Prandtl's Ansatz is false. However, up to now, there is no existence result
for Prandtl equation for small Sobolev perturbations of these shear layer profiles. It is therefore not
possible to correctly define Prandtl boundary corrector in a neighborhood of $u^\nu(0)$.

\item An instability with horizontal wave numbers of order $O(\nu^{-1/8})$, growing like $\exp(C t / \nu^{1/4})$.
These instabilities are much more subtle. Their growth is slower. 
The instability is driven by the so called "critical layer", which
is at a distance $O(\nu^{5/8})$ from the boundary. The current paper is the equivalent of \cite{Gr1} for monotonic profiles.
For these profiles we are not able to prove $O(1)$ separation, but only $O(\nu^{1/4})$. However,
Prandtl equation is well posed in a Sobolev neighborhood of $u^0(\nu)$. We will discuss this limitation
later.

\end{itemize}

In this paper, we shall prove the nonlinear instability of the Ansatz \eqref{Ansatz} near monotonic profiles. 
Roughly speaking, given an arbitrary stable boundary layer, the two main results in this paper are 
\begin{itemize} 

\item in the case of time-dependent boundary layers, we construct Navier-Stokes solutions, with arbitrarily small forcing, of order $\mathcal{O}(\nu^P)$, with $P$ as large as we want,
so that the Ansatz \eqref{Ansatz} is false near the boundary layer, 
up to a remainder of order $\nu^{1/4+\epsilon}$ in $L^\infty$ norm, $\epsilon$ being arbitrarily small. 

\item in the case of stationary boundary layers, we construct Navier-Stokes solutions, without  forcing term, 
so that the Ansatz \eqref{Ansatz} is false, up to a remainder of order $\nu^{5/8}$ in $L^\infty$ norm. 

\end{itemize}

These results prove that there exist no asymptotic expansion of Prandtl's type, even in the case of monotonic profiles. 
For such  profiles, adding viscosity destabilizes the flow, which is counter intuitive.
Even if Prandtl boundary layer equation is well posed, 
it does not describe the limiting behavior as the viscosity goes to $0$.
 
 The complete construction of the instability involves a sublayer, 
 of size $\nu^{5/8}$, which was not expected in this context.
 This sublayer may itself become unstable when it becomes large enough, leading to the creation of a sub - sub - layer.
 
 In this case we are not able to prove that the perturbation reaches $O(1)$ in $L^\infty$ as is the case for 
 Euler unstable profiles, since the linear growth of the perturbation is much slower. In rescaled variables
 (see section $3$), the linear growth is of order $\nu^{1/4}$. A simple equivalent in terms
 of ordinary differential equations would be
 \beq \label{model}
 \dot \phi = \nu^{1/4} \phi + A \phi^2
 \eeq
 where $\nu^{1/4} \phi$ models the linear growth and $\phi^2$ the nonlinear interaction terms.
 For (\ref{model}), the nonlinear term is comparable to the linear one when $\phi$ is of order $\nu^{1/4}$,
 namely very small. The fate of $\phi$ then depends on the sign of $A$ (blow up if $A > 0$ and convergence
 to a $O(\nu^{1/4})$ stable state if $A < 0$). The situation is similar here. The nonlinear term
 is comparable to the linear one when the perturbation reaches $O(\nu^{1/4})$, preventing
 further investigations. 
 
In the next sections, we shall introduce the precise notion of  Rayleigh's stable boundary layers 
and present our main results. 
After a brief recall of the linear instability results \cite{GGN3, GrN3} in Section \ref{sec-linear}, 
we give the proof of the main results in Sections \ref{sec-approx} and \ref{sec-nonlinear}, respectively.


\subsection{Stable boundary layer profiles}\label{sec-BL}


Throughout this paper, by a boundary layer profile, we mean a shear flow of the form 
\begin{equation}\label{def-bl} U_\bl := \begin{pmatrix} U_\bl(t,\frac{y}{\sqrt \nu}) \\0 \end{pmatrix} \end{equation}
that solves the Prandtl's boundary layer problem \eqref{Pr}, with initial data $U_\bl(0,z) = U(z)$. 
Without forcing, $U_\bl$ is the solution of heat equation
$$
\partial_t U_\bl - \partial_{YY} U_\bl =  0 .
$$
Boundary layer profiles can also be generated by adding a forcing term $f^P$, in which case we shall focus 
precisely on the corresponding stationary boundary layers $U_\bl = U(z)$, with $-U''(z) = f^P$.  
We will consider these two different cases, namely time dependent boundary layers (without forcing) and time independent
boundary layers (with given, time independent, forcing).

As mentioned, the Ansatz \eqref{Ansatz} is proven to be false for initial boundary layer profiles $U(z)$ that are spectrally unstable
to the Euler equations \cite{Gr1}.  {\em In this paper, we shall thus focus on stable profiles,} 
those that are spectrally stable to the Euler equations. This includes, for instance, boundary layer profiles without an inflection point 
by view of the classical Rayleigh's inflection point theorem. 
In this paper we assume in addition that $U(z)$ is strictly monotonic, real analytic, that $U(0) =0$ 
and that $U(z)$ converges exponentially fast at infinity to a finite constant $U_+$. 
By a slight abuse of language, such profiles will be referred to
as stable profiles in this paper.

In order to study the instability of such boundary layers, we first analyze the spectrum of the corresponding linearized problem 
around initial profiles $U(z)$. We first introduce the isotropic boundary layer variables $(t,x,z) = (t,x,y)/\sqrt\nu$.
This leads to the following linearized problem for vorticity $\omega = \partial_z v_1 - \partial_x v_2$, which reads 
\begin{equation}\label{EE-vort} 
(\partial_t - L)\omega = 0, \qquad L\omega : = \sqrt\nu \Delta \omega - U\partial_x \omega - v_2 U'' ,
\end{equation}
together with $v = \nabla^\perp \phi$ and $\Delta \phi = \omega$, satisfying the no-slip boundary conditions 
$\phi = \partial_z \phi =0$ on $\{z=0\}$. 

We then take the the Fourier transform in the $x$ variable only, denoting by $\alpha$ the corresponding wavenumber,
which leads to
\begin{equation}\label{EE-vort-a}
 (\partial_t - L_\alpha)\omega = 0, \qquad L_\alpha\omega : = \sqrt\nu \Delta_\alpha \omega - i\alpha U \omega - i\alpha \phi U'' ,
 \end{equation}
where
$$
\omega_\alpha = \Delta_\alpha \phi_\alpha,
$$
together with the zero boundary conditions 
$\phi_\alpha = \phi'_\alpha =0$ on $z=0$. Here, 
$$
\Delta_\alpha = \partial_z^2 - \alpha^2. 
$$
Together with Y. Guo, we proved in \cite{GGN1,GGN3} that, even for profiles $U$ which are stable as $\nu = 0$, 
there are unstable eigenvalues to the Navier-Stokes problem \eqref{EE-vort-a} for sufficiently small viscosity $\nu$ 
and for a range of wavenumber $\alpha \in [\alpha_1,\alpha_2]$, with $\alpha_1\sim \nu^{1/8}$ and $\alpha_2\sim \nu^{1/12}$. 
The unstable eigenvalues $\lambda_*$ of $L_\alpha$, found in \cite{GGN3}, satisfy 
\begin{equation}\label{lambda-GGN} 
\Re \lambda_* \sim \nu^{1/4} .
\end{equation}
Such an instability was first observed by Heisenberg \cite{Hei,HeiICM}, then Tollmien and C. C. Lin \cite{Lin0,LinBook}; 
see also Drazin and Reid \cite{Reid,Schlichting} for a complete account of the physical literature on the subject. 
See also Theorem \ref{theo-spectralinstablity} below for precise details. In coherence with the physical literature \cite{Reid},
we believe that, $\alpha$ being fixed,
 this eigenvalue is the most unstable one. However, this point is an open question from the mathematical point of view.
 
Next, we observe that $L_\alpha$ is a compact perturbation of the Laplacian $\sqrt \nu \Delta_\alpha$, and hence its unstable spectrum in the usual $L^2$ space is discrete. Thus, for each $\alpha,\nu$, we can define the maximal unstable eigenvalue $\lambda_{\alpha,\nu}$ so that $\Re \lambda_{\alpha,\nu}$ is maximum. We set $\lambda_{\alpha,\nu} =0$, if no unstable eigenvalues exist. 

In this paper, we assume that the unstable eigenvalues found in the spectral instability result, Theorem \ref{theo-spectralinstablity}, 
are maximal eigenvalues. Precisely, we introduce 
\begin{equation}\label{def-ga0max0}
\gamma_0 : = \lim_{\nu \to 0} \sup_{\alpha \in \RR}  \nu^{-1/4}\Re \lambda_{\alpha,\nu}  .
\end{equation}
The existence of unstable eigenvalues in Theorem \ref{theo-spectralinstablity} implies that $\gamma_0$ is positive. 
Our spectral assumption is that $\gamma_0$ is finite (that is, the eigenvalues in Theorem \ref{theo-spectralinstablity} are maximal).


\subsection{Main results}


We are ready to state two main results of this paper.

\begin{theorem}\label{theo-approx} 
Let $U_\bl(t,z)$ be a time-dependent stable boundary layer profile as described in Section \ref{sec-BL}. 
Then, for arbitrarily large $s,N$ and arbitrarily small positive $\epsilon$, there exists
a sequence of functions $u^\nu$ that solves the Navier-Stokes equations \eqref{NS1}-\eqref{NS3}, with some forcing $f^\nu$, so that 
$$
\| u^\nu(0) - U_\bl(0) \|_{H^s} + \sup_{t\in [0,T^\nu]} \| f^\nu (t)\|_{H^s} \le \nu^N,
$$
but 
$$
\| u^\nu(T^\nu) - U_\bl(T^\nu) \|_{L^\infty} \ge \nu^{\frac14+\epsilon},
$$
$$\| \omega^\nu(T^\nu) - \omega_\bl(T^\nu) \|_{L^\infty} \to \infty,$$
for time sequences $T^\nu \to 0$, as $\nu \to 0$. Here, $\omega^\nu = \nabla\times u^\nu$ denotes the vorticity of fluids. 
\end{theorem}

This Theorem proves that the Ansatz \eqref{Ansatz} is false, even near stable boundary layers, for data with Sobolev regularity. 

\begin{theorem}[Instability result for stable profiles] \label{firstinstability-stable}
Let $U_\bl = U(z)$ be a stable stationary boundary layer profile as described in Section \ref{sec-BL}. 
Then, for any $s,N$ arbitrarily large, there exists a sequence of solutions $u^\nu$ to the Navier-Stokes equations, 
with forcing $f^\nu = f^P$ (boundary layer forcing), so that $u^\nu$ satisfy 
$$
\| u^\nu(0) - U_\bl \|_{H^s} \le \nu^N,
$$
but 
$$\| u^\nu (T^\nu) - U_\bl\|_{L^\infty} \gtrsim \nu^{5/8},$$
$$\| \omega^\nu (T^\nu) - \omega_\bl \|_{L^\infty} \gtrsim 1,$$
for some time sequences $T^\nu \to 0$, as $\nu \to 0$. 
\end{theorem}

The spectral instability for stable profiles gives rise to sublayers (or critical layers) whose thickness is of order $\nu^{5/8}$. 
The velocity gradient in this sublayer grows like $\nu^{-5/8} e^{t/\nu^{1/4} }$, and becomes larg when $t$ is of order $T^\nu$.
As a consequence, they may in turn become unstable after the instability time $T^\nu$ obtained in the above theorem. 
Thus, in order to improve the $\nu^{5/8}$ instability, one needs to further examine the stability properties 
of this sublayer itself (see \cite{GrN4}).


\subsection{Boundary layer norms}


We end the introduction by introducing the boundary layer norms to be used throughout the paper.
 These norms were introduced in \cite{GrN2} to capture the large, but localized, behavior of vorticity near the boundary.  
 Precisely, for each vorticity function $\omega_\alpha = \omega_\alpha(z)$, we introduce 
the following boundary layer norms
\begin{equation}\label{assmp-wbl-stable} 
\| \omega_\alpha\|_{ \beta, \gamma, 1} : = \sup_{z\ge 0} 
\Big [ \Bigl( 1 +  \delta^{-1} \phi_{p} (\delta^{-1} z)  \Bigr)^{-1} e^{\beta z} |\omega_\alpha (z)| \Big],
\end{equation}
where $\beta > 0$, $p$ is a large, fixed number, 
$$
\phi_p(z) = {1 \over 1 + z^p},
$$
and with the boundary layer thickness 
$$
\delta = \gamma \nu^{1/8}
$$
for some $\gamma>0$. We introduce the boundary layer space ${\cal B}^{\beta,\gamma,1}$ 
  to consist of functions whose $\|\cdot \|_{ \beta, \gamma, 1} $ norm is finite. 
  We also denote by $L^\infty_\beta$ the function spaces equipped with the finite norm 
  $$
  \|\omega \|_\beta = \sup_{z\ge 0} e^{\beta z} |\omega(z)|.
  $$ 
  When there is no weight $e^{\beta z}$, we simply write $L^\infty$ for the usual bounded function spaces. 
  Clearly, 
  $$
  L^\infty_\beta \subset {\cal B}^{\beta,\gamma,1}
.  $$ 
 In addition, it is straightforward to check that 
\begin{equation}\label{al-norm}
\| f g \|_{\beta,\gamma, 1} \le \| f \|_{L^\infty} \| g \|_{ \beta,\gamma,1}.
\end{equation}
Finally, for functions $\omega(x,z)$, we introduce 
$$\| \omega\|_{ \sigma,\beta, \gamma, 1} : = \sup_{\alpha \in \RR} (1+|\alpha|)^{\sigma}\| \omega_\alpha\|_{ \beta, \gamma, 1} ,$$
for $\sigma> 1$, in which $\omega_\alpha$ is the Fourier transform of $\omega$ in the variable $x$. Combining with \eqref{al-norm}, we have 
\begin{equation}\label{al-norm1}
\| f g \|_{\sigma,\beta,\gamma, 1} \le \| f \|_{\sigma,0} \| g \|_{\sigma,\beta,\gamma,1},
\end{equation}
where $\| f \|_{\sigma,0} = \sup_{\alpha \in \RR} (1+|\alpha|)^{\sigma}\| f_\alpha\|_{L^\infty}$. 


\section{Linear instability}\label{sec-linear}


In this section, we shall recall the spectral instability of stable boundary layer profiles \cite{GGN3} and the semigroup estimates 
on the corresponding linearized Navier-Stokes equation \cite{GrN2,GrN3}.
 

\subsection{Spectral instability}\label{sec-grmode}


The following theorem, proved in \cite{GGN3}, provides an unstable eigenvalue of $L$ for generic shear flows.  

\begin{theorem}[Spectral instability; \cite{GGN3}]\label{theo-spectralinstablity}
Let $U(z)$ be an arbitrary shear profile with $U(0)=0$ and $U'(0) > 0$ and satisfy 
$$
\sup_{z \ge 0} | \partial^k_z (U(z) - U_+) e^{\eta_0 z} | < + \infty, \qquad k=0,\cdots ,4,
$$ 
for some constants $U_+$ and $\eta_0 > 0$. Let $R = \nu^{-1/2}$ be the Reynolds number, 
and set $\alpha_\mathrm{low}(R)\sim R^{-1/4}$ and  $ \alpha_\mathrm{up}(R)\sim R^{-1/6}$ be the lower and upper stability branches. 

Then, there is a critical Reynolds number $R_c$ so that for all $R\ge R_c$ and all $\alpha \in (\alpha_\mathrm{low}(R), 
\alpha_\mathrm{up}(R))$, there exist
a nontrivial triple $c(R), \hat v(z; R), \hat p(z;R)$, with $\mathrm{Im} ~c(R) >0$, 
such that $v_R: = e^{i\alpha(x-ct) }\hat v(z;R)$ and $p_R: = e^{i\alpha(x-ct)} \hat p(z;R)$ solve the linearized Navier-Stokes problem 
\eqref{EE-vort}. Moreover there holds the following estimate for the growth rate of the unstable solutions:
$$ 
\alpha \Im c(R) \quad \approx\quad  R^{-1/2}
$$
as $R \to \infty$. 
\end{theorem}

The proof of the previous Theorem, which can be found in \cite{GGN3}, gives a detailed description of the unstable mode.
In this paper we focus on the lower branch of instability.  In this case
$$
\alpha_\nu \approx R^{-1/4} = \nu^{1/8}, \qquad \Re \lambda_\nu\approx R^{-1/2} = \nu^{1/4},
$$
The vorticity of the unstable mode is of the form 
\begin{equation}\label{w0-gr-stable} 
\omega_0 = e^{\lambda_\nu t} \Delta ( e^{i \alpha_\nu x} \phi_0(z) ) \quad + \quad \mbox{complex conjugate} 
\end{equation}
The stream function $\phi_0$ is constructed through asymptotic expansions, and is of the form 
$$
 \phi_0: =  \phi_{in,0}(z) +\delta_\bl \phi_{bl,0} ( \delta_\bl^{-1} z) 
  $$ 
for some boundary layer function $\phi_{bl,0}$, where $\delta_\bl = \nu^{1/8}$.


By construction,  derivatives of $\phi_{\bl,0}$ satisfy 
$$ 
|\partial_z^k \phi_{\bl,0}(\delta_\bl^{-1} z)| \le C_k \delta_\bl^{-k} e^{-\eta_0 z / \delta_\bl} .
$$
%
%
In addition, it is clear that each $x$-derivative of $\omega_0$ gains a small factor of $\alpha_\nu \approx \nu^{1/8}$. 
We therefore have an accurate description of the linear unstable mode.


\subsection{Linear estimates}


The corresponding semigroup $e^{Lt}$ of the linear problem \eqref{EE-vort} is constructed through the path integral 
\beq \label{int2}
e^{L t} \omega = \int_\RR e^{i\alpha x} e^{L_\alpha t} \omega_\alpha \; d\alpha 
\eeq
in which $\omega_\alpha$ is the Fourier transform of $\omega$ in tangential variables and
$L_\alpha$, defined as in \eqref{EE-vort-a}, is the Fourier transform of $L$. 
One of the main results proved in \cite{GrN3} is the following theorem.

\begin{theorem} \cite{GrN3}\label{theo-eLt-stable} 
Let $\alpha \lesssim \nu^{1/8}$.
Let $\omega_\alpha \in {\cal B}^{\beta,\gamma,1}$ for some positive $\beta$, $\gamma_0$ be defined as in \eqref{def-ga0max0}. 
Assume that $\gamma_0$ is finite. 
Then, for any $\gamma_1>\gamma_0$, there is some positive constant $C_\gamma$ so that 
$$\begin{aligned}
\| e^{L_\alpha t}\omega_\alpha\|_{ \beta, \gamma, 1} &\le C_\gamma  
e^{\gamma_1 \nu^{1/4} t }  e^{- \frac14 \alpha^2 \sqrt\nu t}  \| \omega_\alpha\|_{ \beta, \gamma, 1},
\\
\| \partial_ze^{L_\alpha t}\omega_\alpha\|_{ \beta, \gamma, 1} &\le C_\gamma \Big( \nu^{-1/8}+
 (\sqrt \nu t)^{-1/2} \Big) e^{\gamma_1 \nu^{1/4} t }  e^{- \frac14 \alpha^2 \sqrt\nu t}  \| \omega_\alpha\|_{ \beta, \gamma, 1}. 
\end{aligned}$$
\end{theorem}


\section{Approximate solutions}\label{sec-approx}


Let us now construct an approximate solution $u_\app $, which solves Navier-Stokes equations, up to very small error terms. First, we introduce the rescaled isotropic space time variables
$$
\tilde t = {t \over \sqrt{\nu}}, \quad \tilde x = {x \over \sqrt{\nu}}, \quad \tilde z = {z \over \sqrt{\nu}} .
$$
Without any confusion, we drop the tildes. The Navier-Stokes equations in these scaled variables read 
\begin{equation}\label{NS-scaled}
\begin{aligned}
\partial_t u + (u \cdot \nabla) u + \nabla p   &=\sqrt \nu \Delta u,
\\
\nabla \cdot u &= 0,
\end{aligned}\end{equation}
with the no-slip boundary conditions on $z=0$. Theorem \ref{theo-approx} follows at once from the following theorem. 

\begin{theorem}\label{theo-approx-stable} 
Let $U(z)$ be a stable boundary layer profile, and let $U_\bl(\sqrt \nu t,z)$ be the corresponding Prandtl's boundary layer. 
Then, there exist an approximate solution $\widetilde u_\app$ that approximately solves \eqref{NS-scaled} in the following sense: 
for arbitrarily large numbers $p,M$ and for any $\epsilon>0$, the functions $\widetilde u_\app$ 
solve
\begin{equation}\label{NS-Uapp}
\begin{aligned}
\partial_t \widetilde u_\app + (\widetilde u_\app \cdot \nabla) \widetilde u_\app + \nabla \widetilde p_\app   &=
\sqrt \nu \Delta \widetilde u_\app + {\cal E}_\app,
\\
\nabla \cdot \widetilde u_\app &= 0,
\end{aligned}\end{equation}
for some remainder $ {\cal E}_\app$ and for time $t \le T_\nu$, with $T_\nu$ being defined through
$$
\nu^p e^{\Re \lambda_0 T_\nu} = \nu^{\frac14 + \epsilon}.
$$
In addition, for all $t \in [0,T_\nu ]$, there hold
$$
\begin{aligned}\| \mathrm{curl} (\widetilde u_\app - U_\bl(\sqrt \nu t,z) )\|_{\beta,\gamma,1} &\lesssim \nu^{\frac14 + \epsilon},
\\
\|\mathrm{curl} \mathcal{E}_\app (t)\|_{\beta,\gamma,1} 
&\lesssim   \nu^{M} .
\end{aligned}$$
Furthermore, there are positive constants $\theta_0,\theta_1,\theta_2$ independent of $\nu$ so that there holds 
$$  
\theta_1 \nu^p e^{\Re \lambda_0 t} \le  \| (\widetilde u_\app - U_\bl) (t)\|_{L^\infty} \le \theta_2 \nu^p e^{\Re \lambda_0 t} 
$$
for all $t\in [0,T_\nu]$. In particular, 
$$
\| (\widetilde u_\app - U_\bl) (T_\nu)\|_{L^\infty}\gtrsim \nu^{\frac14 + \epsilon}.
$$ 
\end{theorem}


\subsection{Formal construction}


The construction is classical, following \cite{Gr1}. Indeed, the approximate solutions are constructed in the following form
\beq \label{def-Uapp}
\widetilde u_\app(t,x,z) = U_\bl(\sqrt \nu t,z) + \nu^p \sum_{j = 0}^M \nu^{j/8} u_j(t,x,z).
\eeq
For convenience, let us set $v = u -U_\bl $, where $u$ denotes the genuine solution to the Navier-Stokes equations \eqref{NS-scaled}. 
Then, the vorticity $\omega = \nabla \times v$ solves 
$$
 \partial_t \omega + (U_\bl  (\sqrt \nu t, y)  + v)\cdot \nabla \omega + v_2 \partial_y^2U_s(\sqrt \nu t,y) - \sqrt \nu \Delta \omega =0 
$$
in which $v = \nabla^\perp \Delta^{-1} \omega$ and $v_2$ denotes the vertical component of velocity. Here and in what follows, 
$\Delta^{-1}$ is computed with the zero Dirichlet boundary condition. As $U_\bl $ depends slowly on time, 
we can rewrite the vorticity equation as follows: 
\begin{equation}\label{vort-Phi}
 (\partial_t - L) \omega  + \nu^{1/8} S \omega + Q(\omega, \omega) =0.
\end{equation}
In \eqref{vort-Phi}, $L$ denotes the linearized Navier-Stokes operator around the stationary boundary layer $U  = U_s(0,z)$:  
$$ L \omega : = \sqrt \nu \Delta \omega  - U \partial_x \omega - u_2 U'' ,$$
$Q(\omega, \tilde \omega)$ denotes the quadratic nonlinear term $u \cdot \nabla \tilde \omega$, 
with $v = \nabla^\perp \Delta^{-1} \omega$, and $S$ denotes the perturbed operator defined by 
$$
\begin{aligned}
 S\omega: &= \nu^{-1/8}[U_s (\sqrt \nu t, z)  - U (z)] \partial_x \omega + \nu^{-1/8} u_2 [\partial_y^2 U_s (\sqrt \nu t, z)  - U''(z)] 
.\end{aligned}$$
Recalling that $U_s$ solves the heat equation with initial data $U(z)$, we have 
$$| U_s (\sqrt \nu t, z)  - U (z) | \le C \| U''\|_{L^\infty} e^{-\eta_0 z}\sqrt \nu t$$ 
and 
$$ | \partial_z^2 U_s (\sqrt \nu t, z)  - U''(z)| \le C \| U''\|_{W^{2,\infty}} e^{-\eta_0 z}\sqrt \nu t.$$
Hence, 
\begin{equation}\label{def-Sop}
\begin{aligned}
 S\omega  = \nu^{-1/8}\cO(\sqrt \nu t e^{-\eta_0 z})\Big[  |\partial_x \omega |+ |\partial_x \Delta^{-1}\omega| \Big]
\end{aligned}\end{equation}
in which $\Delta^{-1} \omega$ satisfies the zero boundary condition on $z=0$. 
The approximate solutions are then constructed via the asymptotic expansion: 
\begin{equation}\label{def-omegaAAA} \omega_\app = \nu^p \sum_{j=0}^M \nu^{j/8} \omega_j,\end{equation}
in which $p$ is an arbitrarily large integer. Plugging this Ansatz into \eqref{vort-Phi} and matching order in $\nu$,
 we are led to solve 
 \begin{itemize}

\item for $j =0$: 
$$ (\partial_t - L) \omega_0 = 0$$
with zero boundary conditions on $v_0= \nabla^\perp (\Delta)^{-1}\omega_0$ on $z=0$; 

\item for $0<j\le M$: 
\begin{equation}\label{eqs-omegaj} (\partial_t - L) \omega_j  = R_j, \qquad {\omega_j}_{\vert_{t=0}}=0,\end{equation}
with zero boundary condition on $v_j = \nabla^\perp (\Delta)^{-1}\omega_j$ on $z=0$. Here, the remainders $R_j$ are defined by 
$$ 
R_j = S \omega_{j-1} + \sum_{k + \ell + 8p = j}Q(\omega_k, \omega_\ell).
$$

\end{itemize}
As a consequence, the approximate vorticity $\omega_\app$ solves \eqref{vort-Phi}, leaving the error $R_\app$ defined by
\begin{equation}
\label{error-app}
\begin{aligned}
R_\app
&= \nu^{p+\frac{M+1}{8}} S \omega_{M} +  \sum_{k+ \ell> M+1 -8 p; 1\le k,\ell \le M} \nu^{2p+ \frac{k+\ell}{8}}Q(\omega_k, \omega_\ell) 
\end{aligned}\end{equation}
which formally is of order $ \nu^{p+\frac{M+1}{8}}$, for arbitrary $p$ and $M$.

 
\subsection{Estimates}


We start the construction with $\omega_0$ being the maximal growing mode, constructed in Section \ref{sec-grmode}. We recall 
\begin{equation}\label{def-omega000}
\omega_0 = e^ {\lambda_\nu t} e^{i\alpha_\nu x} \Delta_{\alpha_\nu} \Big( \phi_{in,0}(z) 
+\nu^{1/8} \phi_{\bl,0} ( \nu^{-1/8} z) \Big)  \quad+\quad \mbox{c.c.}
\end{equation}
with $\alpha_\nu \sim \nu^{1/8}$ and $\Re \lambda_\nu \sim \nu^{1/4}$. In what follows, $\alpha_\nu$ and $\lambda_\nu$ are fixed. 
We obtain the following lemma.  

\begin{lemma}\label{lem-omegaj} 
Let $\omega_0$ be the maximal growing mode \eqref{def-omega000}, and let $\omega_j$ 
be inductively constructed by \eqref{eqs-omegaj}. Then, there hold the following uniform bounds:
\begin{equation}\label{induction-jn}
\begin{aligned}
\| \partial_x^a\partial_z^b\omega_j \|_{\sigma,\beta,\gamma,1} & \le C_0 \nu^{a/8}\nu^{-b/8} \nu^{-\frac14[\frac{j}{8p}]}
e^{\gamma_0 (1+\frac{j}{8p})\nu^{1/4}t } 
\end{aligned} \end{equation}
for all $a,j\ge 0$ and for $b=0,1$. 
In addition, the approximate solution $\omega_\app$ defined as in \eqref{def-omegaAAA} satisfies 
\begin{equation}\label{est-omegaAAA}
\| \partial_x^a \partial_z^b \omega_\app\|_{\sigma,\beta,\gamma,1} 
\lesssim \nu^{a/8} \nu^{-b/8} \sum_{j=0}^M \nu^{-\frac14[\frac{j}{8p}]}  \Big( \nu^p e^{\gamma_0 \nu^{1/4} t}\Big)^{1+ \frac{j}{8p}} , 
\end{equation}
for $a\ge 0$ and $b=0,1$. 
 Here, $[k]$ denotes the largest integer so that $[k]\le k$. 
\end{lemma}

\begin{proof} 
For $j\ge 1$, we construct $\omega_j$ having the form 
$$ \omega_j = \sum_{n \in \ZZ} e^{i n \alpha_\nu x} \omega_{j,n}$$
It follows that $\omega_{j,n}$ solves  
$$
(\partial_t - L_{\alpha_n}) \omega_{j,n}  = R_{j,n}, \qquad {\omega_{j,n}}_{\vert_{t=0}}=0 
$$
with $\alpha_n = n \alpha_\nu$ and $R_{j,n}$ the Fourier transform of $R_j$ 
evaluated at the Fourier frequency $\alpha_n$. Precisely, we have 
$$ 
R_{j,n} = S_{\alpha_n} \omega_{j-1,n} + \sum_{k + \ell + 8p = j}\sum_{n_1+n_2 = n}Q_{\alpha_n}(\omega_{k,n_1}, \omega_{\ell,n_2}),
$$
in which $S_{\alpha_n}$ and $Q_{\alpha_n}$ denote the corresponding operator $S$ and $Q$ in the Fourier space. 
The Duhamel's integral reads 
\begin{equation}\label{Duh-omegajn}
\omega_{j,n}(t) = \int_0^t e^{L_{\alpha_n} (t-s)} R_{j,n}(s)\; ds \end{equation}
for all $j\ge 1$ and $n \in \ZZ$. 

It follows directly from an inductive argument and the quadratic nonlinearity of $Q(\cdot,\cdot)$ that for all $0\le j\le M$, 
$\omega_{j,n} = 0$ for all $|n|\ge 2^{j+1}$. 
This proves that $|\alpha_n| \le 2^{M+1} \alpha_\nu \lesssim \nu^{1/8}$, for all $|n| \le 2^{M+1}$.  
Since $\alpha_n \lesssim \nu^{1/8}$, the semigroup bounds from Theorem \ref{theo-eLt-stable} read 
\begin{equation}\label{eLt-jn}\begin{aligned}
\| e^{L_\alpha t}\omega_\alpha\|_{ \beta, \gamma, 1} &\lesssim   e^{\gamma_1 \nu^{1/4} t }  
e^{- \frac14 \alpha^2 \sqrt\nu t}  \| \omega_\alpha\|_{ \beta, \gamma, 1},
\\\| \partial_ze^{L_\alpha t}\omega_\alpha\|_{ \beta, \gamma, 1} &\lesssim \Big( \nu^{-1/8}
+ (\sqrt \nu t)^{-1/2} \Big) e^{\gamma_1 \nu^{1/4} t }  e^{- \frac14 \alpha^2 \sqrt\nu t}  \| \omega_\alpha\|_{ \beta, \gamma, 1}. 
\end{aligned}\end{equation}
In addition, since $\alpha_n\lesssim \nu^{1/8}$, from \eqref{def-Sop}, 
we compute 
$$
\begin{aligned}
 S_{\alpha_n}\omega_{j-1,n}  = \cO(\sqrt \nu t e^{-\eta_0 z})\Big[  |\omega_{j-1,n}|+ |\Delta_{\alpha_n}^{-1}\omega_{j-1,n}| \Big]
\end{aligned}
$$
and hence by induction we obtain 
\begin{equation}\label{bd-Sjn}
\begin{aligned}
\| S_{\alpha_n}\omega_{j-1,n}\|_{\beta,\gamma,1} 
&\lesssim \sqrt \nu t \Big[ \|\omega_{j-1,n}\|_{\beta,\gamma,1}+ \| e^{-\eta_0 z}\Delta_{\alpha_n}^{-1}\omega_{j-1,n}\|_{\beta,\gamma,1} \Big]
\\
&\lesssim \sqrt \nu t \Big[ \|\omega_{j-1,n}\|_{\beta,\gamma,1}+ \| \Delta_{\alpha_n}^{-1}\omega_{j-1,n}\|_{L^\infty} \Big]
\\
&\lesssim \sqrt \nu t \nu^{-\frac14[\frac{j-1}{8p}]} e^{\gamma_0 (1+\frac{j-1}{8p})\nu^{1/4}t } ,
\end{aligned}
\end{equation}
 where we used $\|e^{-\eta_0 z} \cdot \|_{\beta,\gamma,1} \le \|\cdot \|_{L^\infty}$ for $\beta < \eta_0$, and 
 $$
 \| \Delta_{\alpha}^{-1} \omega\|_{L^\infty} 
 \le C \| \omega \|_{\beta,\gamma,1},
 $$
uniformly in small $\alpha$; we shall prove this inequality in the Appendix. Let us first consider the case when $1\le j\le 8p-1$, for which $R_{j,n} = S_{\alpha_n} \omega_{j-1,n} $. 
That is, there is no nonlinearity in the remainder. 
Using the above estimate on $S_{\alpha_n}$ and the semigroup estimate \eqref{eLt-jn} into \eqref{Duh-omegajn}, 
we obtain, for $1\le j\le 8p-1$,
$$\begin{aligned} \|\omega_{j,n}(t)\|_{\beta,\gamma,1} 
&\le \int_0^t \| e^{L_{\alpha_n} (t-s)} S_{\alpha_n} \omega_{j-1,n} (s) \|_{\beta,\gamma,1}\; ds 
\\
&\le C\int_0^t e^{\gamma_1 \nu^{1/4} (t-s) }  \| S_{\alpha_n} \omega_{j-1,n} (s) \|_{\beta,\gamma,1}\; ds 
\\
&\le C\int_0^t e^{\gamma_1 \nu^{1/4} (t-s) } \sqrt \nu s e^{\gamma_0 (1+\frac{j-1}{8p})\nu^{1/4}s } \; ds .
\end{aligned}$$
We choose
$$
\gamma_1 = \gamma_0 (1+ \frac{j-1}{8p} + \frac{1}{16p})
$$
in (\ref{eLt-jn}) and use the inequality 
$$
\nu^{1/4} t \le C e^{\frac{\gamma_0}{16p}\nu^{1/4}t}.
$$
and obtain
\begin{equation}\label{est-ojn1}\begin{aligned} \|\omega_{j,n}(t)\|_{\beta,\gamma,1} 
&\le C\int_0^t e^{\gamma_1 \nu^{1/4} (t-s) } \nu^{1/4} e^{\gamma_0 (1+\frac{j-1}{8p} + \frac{1}{16p})\nu^{1/4}s } \; ds 
\\
&\le C\nu^{1/4} e^{\gamma_0 (1+\frac{j-1}{8p} + \frac{1}{16p})\nu^{1/4} t} \int_0^t  \; ds 
\\
&\le C\nu^{1/4} t e^{\gamma_0 (1+\frac{j-1}{8p} + \frac{1}{16p})\nu^{1/4} t} 
\\& \le C e^{\gamma_0 (1+\frac{j}{8p})\nu^{1/4}t } .
\end{aligned}\end{equation}
Similarly, as for derivatives, we obtain 
$$\begin{aligned} 
&
\|\partial_z\omega_{j,n}(t)\|_{\beta,\gamma,1} 
\\&\le \int_0^t \| e^{L_{\alpha_n} (t-s)} S_{\alpha_n} \omega_{j-1,n} (s) \|_{\beta,\gamma,1}\; ds 
\\
&\le C \int_0^t \Big( \nu^{-1/8}+ (\sqrt \nu (t-s))^{-1/2} \Big) e^{\gamma_1 \nu^{1/4} (t-s) }  \| S_{\alpha_n} \omega_{j-1,n} (s) \|_{\beta,\gamma,1}\; ds 
\\
&\le C 
\int_0^t \Big( \nu^{-1/8}+ (\sqrt \nu (t-s))^{-1/2} \Big) e^{\gamma_1 \nu^{1/4} (t-s) }  \sqrt\nu s e^{\gamma_0 (1+\frac{j-1}{8p})\nu^{1/4} s } \; ds ,
\end{aligned}$$
in which the integral involving $\nu^{-1/8}$ is already treated in \eqref{est-ojn1} and bounded by $C\nu^{-1/8}e^{\gamma_0 (1+\frac{j}{8p})\nu^{1/4}t }.$ As for the second integral, we estimate 
\begin{equation}\label{t12-bd} 
\begin{aligned}
\int_0^t &(\sqrt \nu (t-s))^{-1/2} e^{\gamma_1 \nu^{1/4} (t-s) }  \sqrt\nu s e^{\gamma_0 (1+\frac{j-1}{8p})\nu^{1/4} s } \; ds
\\ 
&\le  \int_0^t (\sqrt \nu (t-s))^{-1/2} e^{\gamma_1 \nu^{1/4} (t-s) }  \nu^{1/4} e^{\gamma_0 (1+\frac{j-1}{8p} + \frac{1}{16p})\nu^{1/4} s } \; ds
\\&\le \nu^{1/4} e^{\gamma_0 (1+\frac{j-1}{8p} + \frac{1}{16p})\nu^{1/4} t }  \int_0^t (\sqrt \nu (t-s))^{-1/2} \; ds 
\\&\le C \sqrt t e^{\gamma_0 (1+\frac{j-1}{8p} + \frac{1}{16p})\nu^{1/4} t } 
\\&\le C \nu^{-1/8} e^{\gamma_0 (1+\frac{j}{8p} )\nu^{1/4} t } .
\end{aligned}\end{equation}
Thus, 
$$\|\partial_z\omega_{j,n}(t)\|_{\beta,\gamma,1} \le C \nu^{-1/8} e^{\gamma_0 (1+\frac{j}{8p} )\nu^{1/4} t }.$$ 
This and \eqref{est-ojn1} prove the inductive bound \eqref{induction-jn} for $j\le 8p-1$. 

For $j \ge 8p$, the quadratic nonlinearity starts to play a role. For $k+\ell = j-8p$, we compute 
\begin{equation}\label{est-Q2w}
Q_{\alpha_n}(\omega_{k,n_1}, \omega_{\ell,n_2}) = i \alpha_\nu \Big( n_2 \partial_z \Delta_{\alpha_n}^{-1} \omega_{k,n_1} 
\omega_{\ell,n_2}  - n_1 \Delta_{\alpha_n}^{-1} \omega_{k,n_1} \partial_z \omega_{\ell,n_2}\Big).
\end{equation}
Using the algebra structure of the boundary layer norm (see \eqref{al-norm}), we have 
$$\begin{aligned}
\alpha_\nu\| \partial_z \Delta_{\alpha_n}^{-1} \omega_{k,n_1} \omega_{\ell,n_2} \|_{\beta,\gamma,1} 
&\lesssim 
\nu^{1/8}\| \partial_z \Delta_{\alpha_n}^{-1} \omega_{k,n_1}\|_{L^\infty} \|\omega_{\ell,n_2} \|_{\beta,\gamma,1}
\\&\lesssim 
\nu^{1/8}\| \omega_{k,n_1}\|_{\beta,\gamma,1} \|\omega_{\ell,n_2} \|_{\beta,\gamma,1}
\\&\lesssim \nu^{1/8}\nu^{-\frac14[\frac{k}{8p}]}  \nu^{-\frac14[\frac{\ell}{8p}]}  e^{\gamma_0 (2+\frac{k+\ell}{8p})\nu^{1/4}t }
 \end{aligned}$$
 where we used
 $$
 \| \partial_z \Delta_{\alpha_n}^{-1} \omega_{k,n_1} \|_{L^\infty} 
 \le C \| \omega_{k,n_1}\|_{\beta,\gamma,1},
 $$
 an inequality which is proven in the Appendix.
Moreover,
$$\begin{aligned}
\alpha_\nu\| \Delta_{\alpha_n}^{-1} \omega_{k,n_1} \partial_z \omega_{\ell,n_2} \|_{\beta,\gamma,1} 
&\lesssim \nu^{1/8}
\|  \Delta_{\alpha_n}^{-1} \omega_{k,n_1}\|_{L^\infty} \| \partial_z\omega_{\ell,n_2} \|_{\beta,\gamma,1}
\\&\lesssim 
\nu^{1/8}\| \omega_{k,n_1}\|_{\beta,\gamma,1} \|\partial_z\omega_{\ell,n_2} \|_{\beta,\gamma,1}
\\&\lesssim \nu^{-\frac14[\frac{k}{8p}]}  \nu^{-\frac14[\frac{\ell}{8p}]}  e^{\gamma_0 (2+\frac{k+\ell}{8p})\nu^{1/4}t },
 \end{aligned}$$
 in which the derivative estimate \eqref{induction-jn} was used. 
We note that 
$$
[\frac{k}{8p}]+ [\frac{\ell}{8p}] \le [\frac{k+\ell}{8p}] = [\frac{j}{8p}] - 1.
$$
 This proves 
$$\begin{aligned}
\|& Q_{\alpha_n}(\omega_{k,n_1}, \omega_{\ell,n_2}) \|_{\beta,\gamma,1} 
\lesssim \nu^{1/4} \nu^{-\frac14[\frac{j}{8p}]} e^{\gamma_0 (1+\frac{j}{8p})\nu^{1/4}t }
\end{aligned}$$ 
for all $k+\ell = j-8p$. This, together with the estimate \eqref{bd-Sjn} on $S_{\alpha_n}$, yields 
$$
\begin{aligned}
\| R_{j,n} (t)\|_{\beta,\gamma,1} 
&\lesssim \sqrt \nu t \nu^{-\frac14[\frac{j-1}{8p}]} e^{\gamma_0 (1+\frac{j-1}{8p})\nu^{1/4}t } + \nu^{1/4} \nu^{-\frac14[\frac{j}{8p}]} e^{\gamma_0 (1+\frac{j}{8p})\nu^{1/4}t }
\\
&\lesssim \nu^{1/4} \nu^{-\frac14[\frac{j}{8p}]} e^{\gamma_0 (1+\frac{j}{8p})\nu^{1/4}t },
\end{aligned}
$$
for all $j\ge 8p$ and $n\in \ZZ$, in which we used $\nu^{1/4} t \le e^{\gamma_0 t/8p}$. 

Putting these estimates into the Duhamel's integral formula \eqref{Duh-omegajn}, we obtain, for $j\ge 8p$,
$$
\begin{aligned}
\|\omega_{j,n} (t)\|_{\beta,\gamma,1}  
&\le C \int_0^t  e^{\gamma_1 \nu^{1/4} (t-s) } \| R_{j,n}(s)\|_{\beta,\gamma,1} \; ds
\\
&\le C \int_0^t  e^{\gamma_1 \nu^{1/4} (t-s) } \nu^{1/4} \nu^{-\frac14[\frac{j}{8p}]} e^{\gamma_0 (1+\frac{j}{8p})\nu^{1/4}s } \; ds  
\\
&\lesssim \nu^{-\frac14[\frac{j}{8p}]}  e^{\gamma_0 (1+\frac{j}{8p})\nu^{1/4}s } 
 \end{aligned}$$
and 
$$\begin{aligned} 
&
\|\partial_z\omega_{j,n}(t)\|_{\beta,\gamma,1} 
\\&\le C \int_0^t \Big( \nu^{-1/8}+ (\sqrt \nu (t-s))^{-1/2} \Big) e^{\gamma_1 \nu^{1/4} (t-s) }  \| R_{j,n}(s)\|_{\beta,\gamma,1} \; ds
\\
&\le C 
\int_0^t \Big( \nu^{-1/8}+ (\sqrt \nu (t-s))^{-1/2} \Big) e^{\gamma_1 \nu^{1/4} (t-s) } \nu^{1/4} \nu^{-\frac14[\frac{j}{8p}]} e^{\gamma_0 (1+\frac{j}{8p})\nu^{1/4}s } \; ds .
\end{aligned}$$
Using \eqref{t12-bd}, we obtain 
$$
\begin{aligned}
\|\partial_z \omega_{j,n} (t)\|_{\beta,\gamma,1}  
\lesssim \nu^{-1/8}\nu^{-\frac14[\frac{j}{8p}]}  e^{\gamma_0 (1+\frac{j}{8p})\nu^{1/4}s } ,
 \end{aligned}$$
which completes the proof of \eqref{induction-jn}. The lemma follows. 
\end{proof}

 
 \subsection{The remainder}
 

We recall that the approximate vorticity $\omega_\app$, constructed as in \eqref{def-omegaAAA}, approximately solves \eqref{vort-Phi}, 
leaving the error $R_\app$ defined by
$$
\begin{aligned}
R_\app
&= \nu^{p+\frac{M+1}{8}} S \omega_{M} +  \sum_{k+ \ell> M+1 -8 p; 1\le k,\ell \le M} \nu^{2p+ \frac{k+\ell}{8}}Q(\omega_k, \omega_\ell) .
\end{aligned}$$
Using the estimates in Lemma \ref{lem-omegaj}, we obtain  
$$\begin{aligned}
\|S \omega_{M} \|_{\sigma,\beta,\gamma,1} &\lesssim  \nu^{1/4} \nu^{-\frac14[\frac{M+1}{8p}]} 
e^{\gamma_0 (1+\frac{M+1}{8p})\nu^{1/4}t }
 \\
 \|Q(\omega_k, \omega_\ell) \|_{\sigma,\beta,\gamma,1} &\lesssim  \nu^{1/4} 
 \nu^{-\frac14[\frac{k+\ell}{8p}]} e^{\gamma_0 (2+\frac{k+\ell}{8p})\nu^{1/4}t }.\end{aligned}$$
This yields 
\begin{equation}\label{est-RAA}
\begin{aligned}
\|\ R_\app\|_{\sigma,\beta,\gamma,1} 
&\lesssim  \nu^{1/4} \sum_{j=M+1}^{2M} \nu^{-\frac14[\frac{j}{8p}]}  \Big( \nu^p e^{\gamma_0 \nu^{1/4} t}\Big)^{1+ \frac{j}{8p}} .
\end{aligned}\end{equation}


\subsection{Proof of Theorem \ref{theo-approx-stable}}


The proof of the Theorem now  straightforwardly follows from the estimates from Lemma \ref{lem-omegaj} 
and the estimate \eqref{est-RAA} on the remainder. Indeed, we choose the time $T_*$ so that 
\begin{equation}\label{def-Tstar} \nu^p e^{\gamma_0 \nu^{1/4} T_*} = \nu^\tau\end{equation}
for some fixed $\tau>\frac14$. It then follows that for all $t\le T_*$ and $j\ge 0$, there holds 
$$
\begin{aligned}
 \nu^{-\frac14[\frac{j}{8p}]}  \Big( \nu^p e^{\gamma_0 \nu^{1/4} t}\Big)^{1+ \frac{j}{8p}}  \lesssim \nu^\tau  \nu^{(\tau - \frac14) \frac{j}{8p}}.
\end{aligned}
$$
Using this into the estimates \eqref{est-omegaAAA} and \eqref{est-RAA}, respectively, we obtain 
\begin{equation}\label{est-wRapp}
\begin{aligned}
\| \partial_z^b\omega_\app(t)\|_{\sigma,\beta,\gamma,1} &\lesssim \nu^{p-b/8} e^{\gamma_0 \nu^{1/4} t}  \lesssim \nu^{\tau- b/8},
\\
\|R_\app (t)\|_{\sigma,\beta,\gamma,1} 
&\lesssim  \nu^{1/4} \nu^{-\frac14[\frac{M}{8p}]}  \Big( \nu^p e^{\gamma_0 \nu^{1/4} t}\Big)^{1+ \frac{M}{8p}} 
\\&\lesssim \nu^{\tau+1/4}  \nu^{(\tau - \frac14) \frac{M}{8p}},
\end{aligned}\end{equation}
for all $t \le T_*$. Since $\tau>\frac14$ and $M$ is arbitrarily large (and fixed), the remainder is of order $\nu^P$ 
for arbitrarily large number $P$. The theorem is proved.


\section{Nonlinear instability}\label{sec-nonlinear}


We are now ready to give the proof of Theorem \ref{firstinstability-stable}. Let $\widetilde u_\app$ 
be the approximate solution constructed in Theorem \ref{theo-approx-stable} and let  
$$v = u - \widetilde u_\app,$$
with $u$ being the genuine solution to the nonlinear Navier-Stokes equations. 
The corresponding vorticity $\omega = \nabla \times v$ solves 
$$ 
\partial_t   \omega + (\widetilde u_\mathrm{app} + v) \cdot \nabla \omega + v \cdot \nabla \widetilde \omega_\mathrm{app} 
= \sqrt \nu \Delta \omega +  R_\app
$$
for the remainder $R_\app  = \mathrm{curl } \,  \, {\cal E}_\app$ satisfying the estimate \eqref{est-RAA}. Let us write 
$$
u_\app = \widetilde u_\app - U_\bl .
$$
To make use of the semigroup bound for the linearized operator $\partial_t - L$, we rewrite the vorticity equation as  
$$ 
(\partial_t - L) \omega + (u_\app + v) \cdot \nabla \omega + v \cdot \nabla \omega_\app  = R_\app
$$
with $\omega_{\vert_{t=0}} = 0$. We note that since the boundary layer profile is stationary, 
the perturbative operator $S$ defined as in \eqref{def-Sop} is in fact zero. The Duhamel's principle then yields 
\begin{equation}\label{Duh-omega} 
\omega (t) = \int_0^t e^{L(t-s)} \Big( R_\app - (u_\app + v) \cdot \nabla \omega - v \cdot \nabla \omega_\app  \Big) \; ds.
\end{equation}
Using the representation \eqref{Duh-omega}, we shall prove the existence and give estimates on $\omega$. 
We shall work with the following norm 
\begin{equation}\label{def-123norm} 
||| \omega(t) |||: = \| \omega(t) \|_{\sigma,\beta,\gamma,1}+ \nu^{1/8}\| \partial_x \omega(t) \|_{\sigma,\beta,\gamma,1} 
+ \nu^{1/8}\| \partial_z \omega (t)\|_{\sigma,\beta,\gamma,1} 
\end{equation}
in which the factor $\nu^{1/8}$ added in the norm is to overcome the loss of $\nu^{-1/8}$ for derivatives
(see \eqref{semi-bounds} for more details).

Let $p$ be an arbitrary large number. We introduce the maximal time $T_\nu$ of existence, defined by 
\begin{equation}\label{claim-www}
\begin{aligned}
T_\nu: = \max\Big \{ t\in [0,T_*]~:~\sup_{0\le s\le t} ||| \omega(s) |||  \le \nu^p e^{\gamma_0 \nu^{1/4} t}\Big\}
 \end{aligned}\end{equation}
in which $T_*$ is defined as in \eqref{def-Tstar}. By the short time existence theory, with zero initial data, $T_\nu$ exists and is positive. 
It remains to give a lower bound estimate on $T_\nu$.  First, we obtain the following lemmas. 
\begin{lemma} For $t \in [0,T_*]$, there hold
$$
\begin{aligned}
\| \partial_x^a \partial_z^b \omega_\app(t)\|_{\sigma,\beta,\gamma,1} &\lesssim \nu^{a/8-b/8} \Big( \nu^p e^{\gamma_0 \nu^{1/4} t}\Big)
\\
\|\partial_x^a \partial_z^bR_\app (t)\|_{\sigma,\beta,\gamma,1} 
&\lesssim  \nu^{1/4+a/8-b/8} \nu^{-\frac14[\frac{M}{8p}]}  \Big( \nu^p e^{\gamma_0 \nu^{1/4} t}\Big)^{1+ \frac{M}{8p}} .
\end{aligned}$$
\end{lemma}
\begin{proof} This follows directly from Lemma \ref{lem-omegaj} and the estimate \eqref{est-RAA} on the remainder $R_\app$, 
upon noting the fact that for $t\in [0,T_*]$, $\nu^p e^{\gamma_0 \nu^{1/4} t}$ remains sufficiently small.
\end{proof}

\begin{lemma} There holds 
$$
\Big\|( u_\app + v) \cdot \nabla \omega + v \cdot \nabla  \omega_\app \Big\|_{\sigma,\beta,\gamma,1} 
\lesssim \nu^{-\frac18}\Big( \nu^p e^{\gamma_0 \nu^{1/4} t}\Big)^2.
$$
for $t\in [0,T_\nu]$. 
\end{lemma}
\begin{proof}
We first recall the elliptic estimate 
$$
\| u\|_{\sigma,0} \lesssim \|\omega\|_{\sigma,\beta,\gamma,1}
$$ 
which is proven in the Appendix \ref{sec-elliptic}), and the following uniform bounds (see \eqref{al-norm1})
$$
\begin{aligned}
\| u \cdot \nabla \tilde \omega \|_{\sigma,\beta,\gamma,1}
&\le \| u \|_{\sigma,0} \| \nabla \tilde \omega \|_{\sigma,\beta,\gamma,1}
\\&\le \| \omega \|_{\sigma,\beta,\gamma,1} \| \nabla\tilde \omega \|_{\sigma,\beta,\gamma,1} .
 \end{aligned} $$
Using this and the bounds on $\omega_\app$, we obtain 
$$
\begin{aligned}
 \| v \cdot \nabla \omega_\app \|_{\sigma,\beta,\gamma,1} &\lesssim 
 \nu^{-1/8}\Big( \nu^p e^{\gamma_0 \nu^{1/4} t}\Big) \| \omega \|_{\sigma,\beta,\gamma,1} 
 \lesssim\nu^{-\frac18} \Big( \nu^p e^{\gamma_0 \nu^{1/4} t}\Big)^2
\end{aligned}
$$
and 
$$
\begin{aligned}
 \| ( u_{\app} + v)\cdot \nabla \omega \|_{\sigma,\beta,\gamma,1}  &
 \lesssim \Big( \nu^p e^{\gamma_0 \nu^{1/4} t} + \| \omega\|_{\sigma,\beta,\gamma,1}\Big) \| \nabla \omega \|_{\sigma,\beta,\gamma,1} 
\\&\lesssim \nu^{-\frac18}\Big( \nu^p e^{\gamma_0 \nu^{1/4} t}\Big)^2
 .\end{aligned}
 $$
This proves the lemma. \end{proof}

Next, using Theorem \ref{theo-eLt-stable} and noting that $\alpha e^{-\alpha^2 \nu t} \lesssim 1+ (\nu t)^{-1/2}$, we obtain the following uniform semigroup bounds:
\begin{equation}\label{semi-bounds}\begin{aligned}
\| e^{L t}\omega\|_{ \sigma,\beta, \gamma, 1} &\le C_0 \nu^{-1/4} e^{\gamma_1 \nu^{1/4} t }  \| \omega\|_{\sigma, \beta, \gamma, 1} 
\\
\| \partial_x e^{L t}\omega\|_{ \sigma,\beta, \gamma, 1} &\le
 C_0 \nu^{-1/4}\Big( 1+ (\sqrt \nu t)^{-1/2} \Big) e^{\gamma_1 \nu^{1/4} t }  \| \omega\|_{\sigma, \beta, \gamma, 1} 
\\
\| \partial_z e^{L t}\omega\|_{ \sigma,\beta, \gamma, 1} &\le
 C_0 \nu^{-1/4} \Big( \nu^{-1/8}+ (\sqrt \nu t)^{-1/2} \Big) e^{\gamma_1 \nu^{1/4} t }  \| \omega\|_{\sigma, \beta, \gamma, 1} .
\end{aligned}\end{equation}
We are now ready to apply the above estimates into the Duhamel's integral formula \eqref{Duh-omega}. 
We obtain 
$$
\begin{aligned}
 \|  \omega (t)\|_{\sigma,\beta,\gamma,1} 
 &\lesssim \nu^{-1/4}\int_0^t e^{\gamma_1 \nu^{1/4} (t-s)} \nu^{-\frac18}\Big (\nu^{p} e^{\gamma_0 \nu^{1/4} s}\Big)^2\; ds
 \\&\quad + \nu^{-1/4}\int_0^t e^{\gamma_1 \nu^{1/4} (t-s)}  
 \nu^{1/4} \nu^{-\frac14[\frac{M}{8p}]}  \Big( \nu^p e^{\gamma_0 \nu^{1/4} s}\Big)^{1+ \frac{M}{8p}}  
 \; ds 
\\
 &\lesssim \nu^{-5/8} \Big (\nu^{p} e^{\gamma_0 \nu^{1/4} t}\Big)^2 + \nu^P\Big( \nu^p e^{\gamma_0 \nu^{1/4} t} \Big),
 \end{aligned}$$
upon taking $\gamma_1$ sufficiently close to $\gamma_0$. Set $T_1$ so that 
\begin{equation}\label{def-T11} \nu^p e^{\gamma_0 \nu^{1/4} T_1} =\theta_0 \nu^{\frac58},\end{equation}
for some sufficiently small and positive constant $\theta_0$. 
Then, for all $t\le T_1$, there holds 
$$
 \begin{aligned}\|  \omega (t)\|_{\sigma,\beta,\gamma,1}  
 &\lesssim \nu^{p} e^{\gamma_0 \nu^{1/4}t} \Big[ \theta_0
 + \nu^{P}\Big]
\end{aligned} $$
Similarly, we estimate the derivatives of $\omega$. The Duhamel integral and the semigroup bounds yield 
$$
\begin{aligned}
\| \nabla \omega (t)\|_{\sigma,\beta,\gamma,1} 
 &\lesssim \nu^{-1/4}\int_0^t e^{\gamma_1 \nu^{1/4} (t-s)}  \Big( \nu^{-1/8}+ (\sqrt \nu (t-s))^{-1/2} \Big) 
 \\&\quad \times  
 \Big[ 
 \nu^{-\frac18}\Big (\nu^{p} e^{\gamma_0 \nu^{1/4} s}\Big)^2 
+  \nu^{-\frac14[\frac{M}{8p}]}  \Big( \nu^p e^{\gamma_0 \nu^{1/4} s}\Big)^{1+ \frac{M}{8p}} 
\Big] \; ds 
\\&\lesssim \nu^{-5/8} \Big[ 
 \nu^{-\frac18}\Big (\nu^{p} e^{\gamma_0 \nu^{1/4} s}\Big)^2 
+  \nu^{-\frac14[\frac{M}{8p}]}  \Big( \nu^p e^{\gamma_0 \nu^{1/4} s}\Big)^{1+ \frac{M}{8p}} 
\Big] 
 \end{aligned}$$
By view of \eqref{def-T11} and the estimate \eqref{t12-bd}, 
the above yields 
$$
 \begin{aligned}
 \| \nabla \omega (t)\|_{\sigma,\beta,\gamma,1}  &\lesssim \nu^{p-\frac18} e^{\gamma_0 \nu^{1/4}t} \Big[ \theta_0 
+ \nu^P\Big] .
\end{aligned} 
$$
To summarize, for $t\le \min\{T_*, T_1,T_\nu\}$, with the times defined as in \eqref{def-Tstar},
 \eqref{claim-www}, and \eqref{def-T11}, we obtain 
$$ ||| w(t)||| \lesssim  \nu^{p} e^{\gamma_0 \nu^{1/4}t} \Big[ \theta_0+ \nu^P\Big].$$
Taking $\theta_0$ sufficiently small, we obtain 
$$
 ||| w(t)||| \ll\nu^{p} e^{\gamma_0 \nu^{1/4}t} 
$$
for all time $t\le \min\{T_*, T_1,T_\nu\}$. 
In particular, this proves that the maximal time of existence $T_\nu$ is greater than $ T_1$, defined as in \eqref{def-T11}.  
This proves that at the time $t=T_*$, the approximate solution grows to order of $\nu^{5/8}$ in the $L^\infty$ norm. 
Theorem \ref{firstinstability-stable} is proved.

\appendix


\section{Elliptic estimates}\label{sec-elliptic}


In this section, for sake of completeness, we recall the elliptic estimates with respect to the boundary layer norms. 
These estimates are proven in \cite[Section 3]{GrN2}. 

First, we consider the classical one-dimensional Laplace equation
\beq \label{Lap1}
\Delta_\alpha \phi = \partial_z^2 \phi - \alpha^2 \phi = f
\eeq
on the half line $z \ge 0$, with the Dirichlet boundary condition $ \phi(0) = 0$. 
We recall the function space $L^\infty_\beta$ defined by the finite norm $\| f\|_\beta  =\sup_{z \ge 0} |f(z) | e^{\beta z}$. We will prove
\begin{prop}
If $f \in L^\infty_\beta$ for some $\beta >0$, then $\phi \in L^\infty$. In addition, there holds
\beq \label{Lap3}
(1+\alpha^2) \| \phi \|_{L^\infty} + (1+ | \alpha |) \,  \| \partial_z \phi \|_{L^\infty}
+ \| \partial_z^2 \phi \|_{L^\infty}  \le C \| f \|_{\beta},
\eeq
where the constant $C$ is independent of $\alpha \in \RR$.
\end{prop}
\begin{proof} The solution $\phi$ of (\ref{Lap1}) is explicitly given by
\begin{equation}\label{laplacephi1}
\phi(z) =\int_0^\infty G_\alpha (x,z)f(x) dx 
\end{equation}
where $G_\alpha(x,z) =  - {1 \over 2 \alpha} \Bigl( e^{- \alpha | x-z | }  - e^{-\alpha | x + z |} \Bigr) $. A direct bound leads to
$$
\| \phi \|_{L^\infty} \le {C \over \alpha^2} \| f \|_{\beta}
$$
in which the extra factor of $\alpha^{-1}$ is due to the $x$-integration. 
Differentiating the integral formula, we get
$$
 \|\partial_z \phi \|_{L^\infty} \le {C \over \alpha} \| f \|_{\beta} .
$$
The estimate for $\partial_z^2 \phi$ follows by using directly the equation $\partial_z^2 \phi = \alpha^2 \phi + f$. 
This yields the lemma for the case when $\alpha$ is bounded away from zero. 

As for small $\alpha$, we note that $G_\alpha(0,z) = 0$ and $|\partial_x G_\alpha(x,z)|\le 1$. Hence, $|G_\alpha(x,z)|\le |x|$ and so    
$$ |\phi(z)| \le \int_0^\infty |G_\alpha (x,z)f(x)| dx  \le \|f\|_\beta \int_0^\infty |x| e^{-\beta x} \; dx \le C \| f\|_\beta.$$
Similarly, since $|\partial_z G_\alpha(x,z)|\le 1$, we get 
$$ |\partial_z\phi(z)| \le \int_0^\infty |\partial_zG_\alpha (x,z)f(x)| dx  \le \|f\|_\beta \int_0^\infty e^{-\beta x} \; dx \le C \| f\|_\beta.$$
The lemma follows. 
\end{proof}
We now establish a similar property for ${\cal B}^{\beta,\gamma,1}$ norms:
\begin{prop} \label{proplaplace3}
If $f \in {\cal B}^{\beta,\gamma,1}$ for some $\beta>0$, then $\phi \in L^\infty$. In addition, there holds
\beq \label{Lap4}
(1+| \alpha |) \, \| \phi \|_{L^\infty} 
+  \| \partial_z \phi \|_{L^\infty}   \le C \| f \|_{\beta,\gamma,1},
\eeq
where the constant $C$ is independent of $\alpha \in \RR$. 
\end{prop}
\begin{proof}
We will only consider the case $\alpha > 0$, 
the opposite case being similar. As above, since $G_\alpha(x,z)$ is bounded by $\alpha^{-1}$, using (\ref{laplacephi1}), we have 
$$
| \phi (z) | \le \alpha^{-1} \| f \|_{\beta,\gamma,1} \int_0^\infty
e^{- \alpha |   z -   x |} e^{-\beta   x} 
\Bigl( 1 + \delta^{-1} \phi_P(\delta^{-1} x) \Bigr) d  x
$$
$$
\le  \alpha^{-1} \| f \|_{\beta,\gamma,1} 
\Bigl( \alpha^{-1} + \delta^{-1} \int_0^\infty \phi_P(\delta^{-1} x) d   x \Bigr) 
$$
which yields the claimed bound for $\phi$ since $P >1$. A similar proof applies for $\partial_z \phi$. 
\end{proof}

%

Next, let us now turn to the two dimensional Laplace operator.
\begin{prop} \label{inverseLaplace}
Let $\phi$ be the solution of
$$
- \Delta \phi = \omega
$$
with the zero Dirichlet boundary condition, and let
$$
v = \nabla^\perp \phi. 
$$
If $\omega \in {\cal B}^{\sigma,\beta,\gamma,1}$, then $\phi  \in {\cal B}^{\sigma,0}$ and
$v  = (v_1,v_2) \in {\cal B}^{\sigma,0}$. Moreover, there hold the following elliptic estimates 
\beq \label{Lap4}
\| \phi \|_{\sigma,0} + \|  v_1 \|_{\sigma,0} + \|  v_2 \|_{\sigma,0}   \le C \| \omega \|_{\sigma,\beta,\gamma,1},
\eeq
\end{prop}
\begin{proof}
The proof follows directly from taking the Fourier transform in the $x$ variable, with dual integer Fourier component $\alpha$, and using Proposition \ref{proplaplace3}. 
\end{proof}


%

\begin{thebibliography}{99}




\bibitem{Xu} {  Alexandre, R.;  Wang, Y.-G.;  Xu, C.-J.; and Yang, T.}  Well-posedness of the Prandtl equation in Sobolev spaces.
 {\em J. Amer. Math. Soc.}  28  (2015),  no. 3, 745--784.




\bibitem{Reid} { Drazin, P. G.; Reid, W. H.} {\em Hydrodynamic stability.} Cambridge University Press, 2004.

%
%

\bibitem{DGV} {  G\'erard-Varet, D. and  Dormy, E.}. On the ill-posedness of the Prandtl equation.
{\em J. Amer. Math. Soc.}  23  (2010),  no. 2, 591--609.


\bibitem{DGV2} { G\'erard-Varet, D.,  Maekawa Y., and  Masmoudi, N}
Gevrey Stability of Prandtl Expansions for 2D Navier-Stokes. arXiv:1607.06434

\bibitem{DGV1} {  G\'erard-Varet, D. and  Masmoudi, N}. Well-posedness for the Prandtl system without analyticity or
 monotonicity.  {\em Ann. Sci. \'Ec. Norm. Sup\'er. } (4)  48  (2015),  no. 6, 1273--1325.


\bibitem{GVN} { D. G\'erard-Varet and T. Nguyen}. Remarks on the ill-posedness of the Prandtl equation.
{\em Asymptotic Analysis,} 77 (2012), no. 1-2, 71--88.


\bibitem{Gr1} { Grenier, E.}
\newblock On the nonlinear instability of {E}uler and {P}randtl equations.
\newblock {\em Comm. Pure Appl. Math. 53}, 9 (2000), 1067--1091.


\bibitem{GGN1}
{E. Grenier, Y. Guo, and T. Nguyen.}
\newblock Spectral stability of {P}randtl boundary layers: an overview.
\newblock {\em Analysis (Berlin)}, 35(4):343--355, 2015.

\bibitem{GGN3}
{ E. Grenier, Y. Guo, and T. Nguyen.}
\newblock Spectral instability of characteristic boundary layer flows.
\newblock {\em Duke Math J., to appear}, 2016.

%


\bibitem{GrN2} { E. Grenier and T. Nguyen,} Sharp bounds on linear semigroup of Navier-Stokes with boundary layer norms. arXiv:1703.00881

\bibitem{GrN3} { E. Grenier and T. Nguyen,} Green function for linearized Navier-Stokes around a boundary layer profile: near critical layers. Preprint 2017.  

\bibitem{GrN4} { E. Grenier and T. Nguyen,} Sublayer of Prandtl boundary layers. {\em Arch. Ration. Mech. Anal. } to appear 2018. 


\bibitem{Hei} { Heisenberg, W.} $\ddot{\mbox{U}}$ber Stabilit$\ddot{\mbox{a}}$t und Turbulenz von Fl$\ddot{\mbox{u}}$ssigkeitsstr$\ddot{\mbox{o}}$men. Ann. Phys. 74, 577--627 (1924)


\bibitem{HeiICM} { Heisenberg, W.}  {\em On the stability of laminar flow.}
 Proceedings of the International Congress of Mathematicians, Cambridge,
 Mass., 1950, vol. 2, 
 pp. 292--296. Amer. Math. Soc., Providence, R. I.,  1952. 



\bibitem{GuoN} {  Y. Guo and T.  Nguyen, }  A note on Prandtl boundary layers.
{\em Comm. Pure Appl. Math.}  64  (2011),  no. 10, 1416--1438.



\bibitem{VV} {  Ignatova, M.  and Vicol, V.} Almost global existence for the Prandtl boundary layer equations.
{\em Arch. Ration. Mech. Anal. } 220  (2016),  no. 2, 809--848.



%
%
%

 
\bibitem{Lin0} { C. C. Lin}, {On the stability of two-dimensional parallel flow}, Proc. Nat. Acad. Sci. U. S. A.  30,  (1944). 316--323.



\bibitem{LinBook} { C. C. Lin}, {\em The theory of hydrodynamic stability.} {Cambridge, at the University Press, 1955.}
%
%

\bibitem{Mae} {  Maekawa, Y.} On the inviscid limit problem of the vorticity equations for viscous
 incompressible flows in the half-plane.
 {\em Comm. Pure Appl. Math.}  67  (2014),  no. 7, 1045--1128.


 
\bibitem{MW} {  Masmoudi, N. and  Wong, T. K.} Local-in-time existence and uniqueness of solutions to the Prandtl
 equations by energy methods.
{\em Comm. Pure Appl. Math.}  68  (2015),  no. 10, 1683--1741.



\bibitem{Oleinik1} {  Oleinik, O. A.}  On the mathematical theory of boundary layer for an unsteady flow of
 incompressible fluid.
{\em  Prikl. Mat. Meh.}  30 801--821 (Russian); translated as {\em J. Appl. Math. Mech.}  30  1966 951--974 (1967).

\bibitem{Oleinik} {  Oleinik, O. A. and  Samokhin, V. N.} {\em Mathematical models in boundary layer theory.}
Applied Mathematics and Mathematical Computation, 15. Chapman \& Hall/CRC, Boca Raton, FL,  1999. x+516 pp. ISBN: 1-58488-015-5 



\bibitem{Caf1} {  Sammartino, M. and Caflisch, R. E.}  Zero viscosity limit for analytic solutions, of the Navier-Stokes
 equation on a half-space. I. Existence for Euler and Prandtl equations.
{\em Comm. Math. Phys.}  192  (1998),  no. 2, 433--461.
 
\bibitem{Caf2} {  Sammartino, M. and  Caflisch, R. E.}  Zero viscosity limit for analytic solutions of the Navier-Stokes
 equation on a half-space. II. Construction of the Navier-Stokes solution.
 {\em Comm. Math. Phys.}  192  (1998),  no. 2, 463--491.
 


\bibitem{Schlichting} { H. Schlichting,} {\em Boundary layer theory,} 
Translated by J. Kestin. 4th ed. McGraw--Hill Series in Mechanical Engineering. McGraw--Hill Book Co., Inc., New York, 1960.



\end{thebibliography}
\end{document}